\documentclass[11pt]{article}

\usepackage{graphicx}
\usepackage{caption}
\usepackage{subcaption}
\usepackage{amsmath,amsfonts,amssymb,amsthm}
\usepackage{url}
\usepackage{dsfont} 
\usepackage{enumerate}
\usepackage{color}

\usepackage[nodayofweek]{datetime}

\usepackage[backend=bibtex,bibstyle=authoryear,citestyle=authoryear,maxbibnames=3,uniquename=init,giveninits]{biblatex}
\DeclareNameAlias{sortname}{last-first}
\DeclareNameAlias{default}{last-first} 
\renewbibmacro{in:}{}
\usepackage[top=25mm]{geometry}

\newtheorem{theorem}{Theorem}
\newtheorem{lemma}{Lemma}

\theoremstyle{definition}
\newtheorem{remark}{Remark}
\newtheorem*{remark*}{Remark}
\newtheorem*{assumption*}{Assumption}
\newtheorem*{assumptions*}{Assumptions}

\newcommand{\indic}{\mathds 1}
\renewcommand{\P}{\mathbb{P}}
\newcommand{\E}{\mathbb{E}}

\newcommand{\dd}{{\rm d}}
\newcommand{\dx}{\dd x}
\newcommand{\dy}{\dd y}

\renewcommand{\tilde}{\widetilde}

\renewcommand{\t}{\theta}
\renewcommand{\l}{\lambda}
\newcommand{\Fbar}{\overline F}
\newcommand{\FIbar}{\overline F_I}

\newcommand{\FItilbar}{\overline {\tilde F_I}}
\newcommand{\psitil}{\tilde\psi}
\newcommand{\Ttil}{\tilde T}
\newcommand{\ttil}{\tilde\t}
\newcommand{\ltil}{\tilde\l}
\newcommand{\ctil}{\tilde c}
\newcommand{\mtil}{\tilde m}
\newcommand{\Ftil}{\tilde F}
\newcommand{\fItil}{\tilde f_I}
\newcommand{\FItil}{\tilde F_I}
\newcommand{\Util}{\tilde U}
\newcommand{\Mtil}{\tilde M}
\newcommand{\Ztil}{\tilde Z}
\newcommand{\Rtil}{\tilde R}

\renewcommand{\O}{\mathcal O}
\newcommand{\aux}[1]{#1^\textrm{aux}}
\newcommand{\DV}[1]{#1^\textrm{DV}}

\newcommand{\psiDV}{\DV{\psi}}
\newcommand{\psiaux}{\aux{\psi}}

\newcommand{\psiCL}{\psi^\textrm{CL}}
\newcommand{\GRbar}[1]{\overline G_{#1}}
\newcommand{\mexp}{m^\textrm{exp}_k}
\newcommand{\RDV}{\DV{R}}
\newcommand{\mDV}{\DV{m}}
\newcommand{\tDV}{\DV{\t}}
\newcommand{\dDV}{\DV{d}}
\newcommand{\ttDV}{\DV{t}}
\newcommand{\lDV}{\DV{\lambda}}
\newcommand{\sDV}{\DV{s}}
\renewcommand{\cite}{\textcite}

\begin{document}

\title{Error estimates for De Vylder type approximations in ruin theory}


\date{\today}

\author{Azmi Makhlouf\footnote{Universit\'e de Tunis El Manar (UTM), Ecole Nationale d'Ing\'enieurs de Tunis (ENIT), LR99ES20 Laboratoire de Mod\'elisation Math\'ematique et Num\'erique dans les Sciences de l'Ing\'enieur (LAMSIN), B.P. 37, 1002 Tunis, Tunisia. E-mail: azmi.makhlouf@enit.utm.tn}}
\maketitle

\begin{abstract}
Due to its practical use, De Vylder's approximation of the ruin probability has been one of the most popular approximations in ruin theory and its application to insurance. Surprisingly, only heuristic and numerical evidence has supported it, to some extent. Finding a mathematical estimate for its accuracy has remained an open problem, going from the original paper by De Vylder (1978) through an attempt of justification by Grandell (2000).\\ 
The present paper consists of a mathematical and critical treatment of the problem. We more generally consider De Vylder type approximations of any order $k$, based on fitting the $k$ first moments of the classical risk reserve process. Moreover, we not only deal with the ruin probability, but also with the moments of the time of ruin, of the deficit at ruin and of the surplus before ruin.\\
We estimate the approximation errors in terms of the safety loading coefficient, the initial reserve and the approximation order. We show their different behaviours, and the extent to which each relative error remains small or blows up, so that one has to be careful when using this approximation. Our estimates are confirmed by numerical examples.\\
Besides, it turns out that De Vylder type approximations become paradoxically inaccurate when applied to the moments of the deficit at ruin and of the surplus before ruin.
\end{abstract}
\textbf{Keywords}: Ruin theory, De Vylder approximation; Error estimates; Ruin probability; Time of ruin; Deficit; Surplus\\

\section{Introduction}

One of the major interests of both the actuarial theory and practice is the ruin event, that is when an insurer's risk reserve, subject to random claims arrivals, becomes negative.\\
According to the classical compound Poisson model, the reserve at time $t$ is given by
\begin{equation}\label{eq:reserve}
U_t=U_t(u)=u+ct-\sum_{i=1}^{N_t}Z_i,
\end{equation}
where $u=U_0\geq 0$ is the initial reserve, $c>0$ the premium per unit time, $(N_t)_{t\geq 0}$ (number of claims until time $t$) a Poisson process of intensity $\l$, and $(Z_i)_{i\geq 1}$ positive independent random variables (costs of the claims).\\
Among the most important concepts that are related to the ruin event (and that may be viewed as risk indicators), we mention the ruin probability, the time of ruin, the deficit at ruin and the surplus before ruin.\\
Now, although the basic statistical properties of the above quantities are known to satisfy some renewal equations (see \cite{lin:will:99} and \cite{lin:will:00}), the solution is not explicit in general, and numerical approximations are usually needed.\\
Several approximations for the ruin probability have been proposed (see \cite{asmu:albr:10} and \cite{cize:hard:wero:11}, for instance). For obvious practical reasons, much attention has been paid to so-called simple approximations, that are fully explicit and that only use few moments of the claim distribution. For a detailed account of such approximations, we refer the reader to \cite{gran:00}.\\
In the present paper, we are interested in De Vylder's approximation (\cite{de:vyld:78}), which has been one of the most popular simple approximations.\\ 
De Vylder's original idea is to match the three first moments of the risk reserve $(U_t)$ with those of a risk reserve $(\Util_t)$ where the claims are exponentially distributed, taking advantage of the explicit expression of the ruin probability in this case.\\
\cite{de:vyld:78} only gave numerical results, that surprisingly showed the efficiency of this approximation with several examples of light-tailed claim distributions, on some range of the parameters.\\
Later, \cite{gran:00} suggested a possible mathematical explanation, by considering some parameters of the model. Indeed, he showed that De Vylder's method approximates the adjustment coefficient (Lundberg exponent) at a cubic rate with respect to the safety loading coefficient, which is often small (smaller than 1 in practice). However, the author mentioned that he did not manage to derive an estimate for the approximation error of the ruin probability, which has (to the best of our knowledge) remained an open problem.\\
De Vylder's method has been extended in different directions.\\
\cite{dick:wong:04} applied it to approximate the moments and density of the time of ruin.\\
\cite{burn:mist:wero:05} proposed a four-moment Gamma approximation of the ruin probability, instead of the original three-moment exponential one, and numerically observed an improvement of the accuracy.\\
In this paper, we more generally consider matching the $k$ first moments, with any order $k\geq 2$. Besides, we need neither explicit expressions for the approximating ruin probability nor for the parameters of the approximating model (we will rather rely on their particular structure given by Lemma \ref{lem:U:Util}). Moreover, we also deal with De Vylder type approximations of the moments of the time of ruin, the deficit at ruin and the surplus before ruin.\\
We aim at mathematically estimating the approximation errors, with respect to the safety loading coefficient, the initial reserve and the approximation order $k$.
We state upper bounds for the relative errors, expressing the rate at which they are either small or large.\\
Numerical examples illustrate our estimates. We point out how the errors may blow-up (even when the parameters are of practical interest), and that De Vylder type approximations applied to the moments of the deficit at ruin and of the surplus before ruin are surprisingly inaccurate.\\
The paper is organized as follows. In Section \ref{sec:nota:assump}, we introduce additional notations and assumptions. In Section \ref{sec:key:lemmas}, we establish some key lemmas. In Sections \ref{sec:err:proba}, \ref{sec:err:time}, \ref{sec:err:deficit} and \ref{sec:err:surplus}, we state and prove our main results (given by Theorems \ref{thm:DV:err:proba}, \ref{thm:DV:err:time}, \ref{thm:DV:err:deficit} and \ref{thm:DV:err:surplus}). Finally, we conclude in Section \ref{sec:concl}.

\section{Notations and assumptions}\label{sec:nota:assump}
The costs of the claims $(Z_i)_{i\geq 1}$ are positive independent random variables with the same cumulative distribution function $F$ of unbounded support. We denote the $j^{\rm th}$ moment of the claim cost by
$$m_j:=\E(Z_1^j).$$
For every cumulative distribution function $\Phi$, the survival function $1-\Phi$ is denoted by $\overline \Phi$.\\
The integrated tail distribution function is defined by
$$F_I(x)=\frac{1}{m_1}\int_0^x\Fbar(y)\dy,$$
and the associated density function is
$$f_I(x)=\frac{\Fbar(x)}{m_1}.$$
The time of ruin is
$$T=T(u):=\inf\{t\geq 0: U_t<0\},$$
and the ruin probability is $$\psi(u):=\P(T(u)<\infty).$$
Its derivative (with respect to $u>0$) is denoted by $\psi'(u)$.\\
The safety loading coefficient is defined by
$$\t:=\frac{c}{\lambda m_1}-1,$$
which is usually supposed to be positive in order to avoid an almost sure ruin. Let us mention here that the dependence of $\psi(u)$ on $\t$ is just made implicit.\\
Our study lies within the framework of light-tailed claim distributions. We shall use the following assumptions. 
\begin{assumptions*}
\ \vspace{0.2em}
\begin{enumerate}[(\textrm{A}1)]
\item\label{assump:CL} There exists $R>0$ (the adjustment coefficient, that depends on $\t$) such that
\begin{equation}\label{eq:def:R}
\int_0^\infty\exp(Rx)\frac{f_I(x)}{1+\t}\dx=1.
\end{equation}
\item\label{assump:x:exp}
\begin{equation*}
	\mexp:=\int_0^{\infty}x^k\exp(Rx)\frac{f_I(x)}{1+\t}\dx<\infty.
\end{equation*}
\item\label{assump:hazard} The hazard rate function associated to the tail distribution is bounded:
$$h_I:=\sup_{x\geq 0}\frac{f_I(x)}{\FIbar(x)}<\infty.$$
\end{enumerate}
\end{assumptions*}
Assumption (A\ref{assump:CL}) is the usual Cramer-Lundberg condition. We will need Assumption (A\ref{assump:x:exp}) when considering the approximation errors. We shall use assumption (A\ref{assump:hazard}) only for estimating the derivative of the ruin probability (proof of Lemma \ref{lem:psi:psitil}). It is satisfied by many classic claim distributions including exponential, mixture of exponentials, Gamma, and inverse Gaussian distributions (see \cite{klup:89}).\\

For any functions $a$ and $b$, we denote the convolution operation by 
\begin{equation*}
a*b(u):=\int_0^u a(u-x)b(x)\dx.
\end{equation*}
Throughout the paper, the notation
\begin{equation*}
a(u,\t)=\O\left(b(u,\t)\right) \textrm{ (or } a=\O(b)\textrm{)}
\end{equation*}
means that there exists a non-negative constant $C$, that does not depend on $(u,\t)$ (but may depend on the $m_j$'s, $\mexp$ and $h_I$), such that $|a(u,\t)|\leq C|b(u,\t)|$ for all $u\geq 0$ and $\t>0$, with $\t$ lying in a bounded interval (in practice, $\t$ is usually smaller than unity).\\
The notation $x\wedge y$ stands for $\min(x,y)$; $x\vee y$ for $\max(x,y)$, and $\indic_{x>y}$ stands for 1 if $x>y$ and 0 otherwise.\\
If $i_1>i_2$, $\sum_{i=i_1}^{i_2}a_i:=0$.\\
For any quantity $X$ and an approximation $X^\textrm{approx}$ of it, the corresponding relative error is defined by
$$\varepsilon^\textrm{approx}:=\left\vert\frac{X-X^\textrm{approx}}{X}\right\vert.$$

\section{Key lemmas}\label{sec:key:lemmas}

In this section, we state and prove some estimates that we will need for the proofs of our main results. Besides, they have an interest for their own.

\subsection{Results for general approximations}

We consider a second compound Poisson reserve model (that stands for any approximating model)
\begin{equation}\label{eq:reserve:tilde}
\Util_t=\Util_t(u)=u+\tilde ct-\sum_{i=1}^{\tilde N_t}\tilde Z_i,
\end{equation}
with parameters $(\ctil, \ltil, \mtil_j, \ttil, \Ftil, \FItil, \fItil, \Rtil, \psitil, \Ttil)$ instead of $(c, \l,  m_j,\t, F, F_I, f_I,R,\psi,T)$.\\
In the following lemma, we give an estimate of the difference between ruin probabilities, with no mention of any particular approximation procedure. 
\begin{lemma}\label{lem:psi:psitil}
Let $\psi(u)$ and $\psitil(u)$ be the ruin probabilities associated to models \eqref{eq:reserve} and \eqref{eq:reserve:tilde} respectively, under assumptions (A\ref{assump:CL}) and (A\ref{assump:hazard}) for both models. Set 
\begin{equation*}
\Delta(u):=\psi(u)-\psitil(u)
\end{equation*}
and
\begin{equation*}
\delta(u):=\frac{\FIbar(u)}{1+\t}-\frac{\FItilbar(u)}{1+\ttil}+\psitil*(\frac{f_I}{1+\t}-\frac{\fItil}{1+\ttil})(u).
\end{equation*}
Then, for all $u\geq 0$,
\begin{equation}
\vert\Delta(u)\vert \leq \left(\vert\delta(u)\vert\exp(Ru)+h_I\int_0^u\vert\delta(x)\vert\exp(Rx)\dx\right)\exp(-Ru).
\end{equation}
\end{lemma}

\begin{proof}
First, we prove that
\begin{equation}\label{eq:Delta:psiprim}
\Delta(u)=\delta(u)-\frac{1+\t}{\t}\psi'*\delta(u).
\end{equation}
It is well known (see, for example, \cite{fell:66} or \cite{gerb:79}) that $\psi$ solves the renewal equation
\begin{equation}\label{eq:psi}
\psi(u)=\frac{\FIbar(u)}{1+\t}+\frac{1}{1+\t}\psi*f_I(u).
\end{equation}
Similarly,
\begin{equation*}
\psitil(u)=\frac{\FItilbar(u)}{1+\ttil}+\frac{1}{1+\ttil}\psitil*\fItil(u).
\end{equation*}
Thus,
\begin{equation}\label{eq:Delta:renewal}
\Delta(u)=\delta(u)+\frac{1}{1+\t}\Delta*f_I(u).
\end{equation}
This is a renewal equation in $\Delta$, and actually, \eqref{eq:Delta:psiprim} follows from \cite[Theorem 2.1]{lin:will:99}, which is based on the use of a compound geometric series and Laplace transform. For the convenience of the reader, we give a more simple argument.\\
By uniqueness of the solution of the renewal equation \eqref{eq:Delta:renewal}, it is sufficient to check that $\Delta$ given by \eqref{eq:Delta:psiprim} solves it. Set
$$\Delta_{\eqref{eq:Delta:psiprim}}:=\delta-\frac{1+\t}{\t}\psi'*\delta.$$
First, it is known that $\psi$ is differentiable on $(0,\infty)$ and that (from \eqref{eq:psi})
\begin{align}
\psi'(u)&=-\frac{f_I(u)}{1+\t}+\frac{1}{1+\t}\psi(0)f_I(u)+\frac{1}{1+\t}\psi'*f_I(u)\nonumber\\
&=-\frac{\t}{(1+\t)^2}f_I(u)+\frac{1}{1+\t}\psi'*f_I(u).\label{eq:psiprim}
\end{align}
Then,
\begin{align*}
\Delta_{\eqref{eq:Delta:psiprim}}&=\delta-\frac{1+\t}{\t}\delta*\left(-\frac{\t}{(1+\t)^2}f_I+\frac{1}{1+\t}\psi'*f_I\right)\\
&=\delta+\frac{1}{1+\t}\delta*f_I-\frac{1}{\t}\delta*\psi'*f_I\\
&=\delta+\frac{1}{1+\t}(\delta-\frac{1+\t}{\t}\delta*\psi')*f_I\\
&=\delta+\frac{1}{1+\t}\Delta_{\eqref{eq:Delta:psiprim}}*f_I.
\end{align*}
Thus $\Delta=\Delta_{\eqref{eq:Delta:psiprim}}$.
We then deduce Lemma \ref{lem:psi:psitil} using the following estimate of the derivative of the ruin probability, which solves renewal equation \eqref{eq:psiprim} (see \cite[Corollary 3.1 and Example 5.3]{will:cai:lin:01}):
$$\vert\psi'(x)\vert=-\psi'(x)
\leq \frac{h_I\t}{(1+\t)^2}\exp(-Rx).$$
\end{proof}

The following lemma gives the order of Lundberg type approximation of the ruin probability with respect to the safety loading coefficient $\t>0$ and for all $u\geq 0$.

\begin{lemma}\label{lem:psi:exp:err}
	Suppose that the reserve model $(U_t)$ satisfies assumptions (A\ref{assump:CL}), (A\ref{assump:x:exp}) and (A\ref{assump:hazard}).\\
	Then, 
	\begin{equation*}
	\psi(u)-\exp(-Ru)=\O(\t)\exp(-Ru).
	\end{equation*}
\end{lemma}

As a matter of fact, Lemma \ref{lem:psi:exp:err} will give us an estimate for the Cramer-Lundberg approximation error with respect to $\t$ and $u$ (see section \ref{sec:other:approx}). Note, however, that Lemma \ref{lem:psi:exp:err} holds for all $u\geq 0$, unlike the Cramer-Lundberg approximation, which is only an asymptotic result for large $u$, and that the rate $\O(\t)$ is optimal as easily checked for $u=0$.\\

\begin{proof}[Proof of Lemma \ref{lem:psi:exp:err}]
	Let the superscript "aux" denote an auxiliary approximation that we define by
	\begin{equation}\label{eq:def:psiaux}
	\aux{\psi}(u):=\frac{1}{1+\t}\exp(-Ru).
	\end{equation}
	Since $\aux{\psi}(u)$ can be viewed as a ruin probability associated to a model with exponential claims, it is easy to check that
	\begin{equation*}
	\psiaux=\frac{\aux{\FIbar}}{1+\t}+\frac{1}{1+\t}\psiaux*\aux{f_I},
	\end{equation*}
	with $$\aux{f_I}(x)=\mu\aux{\FIbar}(x)=\mu\exp(-\mu x), \textrm{ where }\mu=\frac{R(1+\t)}{\t},$$
	and that
	\begin{equation}\label{eq:R:aux}
	\int_0^\infty\exp(Rx)\frac{\aux{f_I}(x)}{1+\t}\dx=1.
	\end{equation}
	Set $\aux{\Delta}:=\psi-\psiaux$. Using Lemma \ref{lem:psi:psitil}, we have
	\begin{equation}\label{eq:Delta:delta:aux:estim}
	\vert\aux{\Delta}(u)\vert \leq \vert\aux{\delta}(u)\vert+h_I\exp(-Ru)\int_0^u\exp(Rx)\vert\aux{\delta}(x)\vert\dx,
	\end{equation}
	where
	\begin{align*}
	\aux{\delta}(u)&=\frac{\FIbar(u)}{1+\t}-\frac{\aux{\FIbar}(u)}{1+\t}+\psiaux*(\frac{f_I}{1+\t}-\frac{\aux{f_I}}{1+\t})(u)\\
	&=\frac{\FIbar(u)}{1+\t}-\frac{\aux{\FIbar}(u)}{1+\t}+\frac{\exp(-Ru)}{1+\t}\int_0^u\exp(Rx)(\frac{f_I(x)}{1+\t}-\frac{\aux{f_I}(x)}{1+\t})\dx.
	\end{align*}
	Set $$\aux{\widehat\delta}(u):=\aux{\delta}(u)\exp(Ru).$$
	By \eqref{eq:def:R} and \eqref{eq:R:aux}, we get
	\begin{align*}
	\aux{\widehat\delta}(u)&=\exp(Ru)\left(\frac{\FIbar(u)}{1+\t}-\frac{\aux{\FIbar}(u)}{1+\t}\right)-\frac{1}{1+\t}\int_u^\infty\exp(Rx)\left(\frac{f_I(x)}{1+\t}-\frac{\aux{f_I}(x)}{1+\t}\right)\dx\\
	&=\exp(Ru)\left(\frac{\FIbar(u)}{1+\t}-\frac{\aux{\FIbar}(u)}{1+\t}\right)\left(1-\frac{1}{1+\t}\right)\\
	&-\frac{1}{1+\t}\int_u^\infty\Big(\exp(Rx)-\exp(Ru)\Big)\left(\frac{f_I(x)}{1+\t}-\frac{\aux{f_I}(x)}{1+\t}\right)\dx.
	\end{align*}
	Set, for $u\geq 0$ and $p\in\{0,1\}$,
	\begin{equation*}
	\GRbar{p}(u):=\int_u^\infty x^p \exp(Rx)\frac{f_I(x)}{1+\t}\dx
	\end{equation*}
	and
	\begin{equation*}
	\aux{\GRbar{p}}(u):=\int_u^\infty x^p \exp(Rx)\frac{\aux{f_I}(x)}{1+\t}\dx.
	\end{equation*}
	Then,
	\begin{align}
	\vert\aux{\widehat\delta}(u)\vert &\leq \t\exp(Ru)\left(\frac{\FIbar(u)}{1+\t}+\frac{\aux{\FIbar}(u)}{1+\t}\right)\nonumber\\
	&+R\int_u^\infty (x-u)\exp(Rx)\left(\frac{f_I(x)}{1+\t}+\frac{\aux{f_I}(x)}{1+\t}\right)\dx\nonumber\\
	&\leq \t\left(\GRbar{0}(u)+\aux{\GRbar{0}}(u)\right)+R\left(\GRbar{1}(u)+\aux{\GRbar{1}}(u)\right)\nonumber\\
	&=\O(\t)\Big(\GRbar{0}(u)+\aux{\GRbar{0}}(u)+\GRbar{1}(u)+\aux{\GRbar{1}}(u)\Big),\label{eq:delta:aux:hat:G}
	\end{align}
	since $R=\O(\t)$ (see \cite{gran:00} or also \eqref{eq:R:O:theta}).\\
	Clearly, by assumption (A\ref{assump:x:exp}), both $\GRbar{0}$ and $\GRbar{1}$ are $\O(1)$. It is also easy to check (with the use of \eqref{eq:R:O:theta}) that both $\aux{\GRbar{0}}$ and $\aux{\GRbar{1}}$ are $\O(1)$. Thus, from \eqref{eq:delta:aux:hat:G},
	\begin{equation}\label{eq:delta:aux:O:theta}
	\aux{\widehat\delta}(u)=\O(\t).
	\end{equation}
	Besides, for $p\in\{0,1\}$,
	\begin{align}
	\int_0^\infty\GRbar{p}(x)\dx&=\int_0^\infty\int_x^\infty y^p\exp(Ry)\frac{f_I(y)}{1+\t}\dy\dx\nonumber\\
	&=\int_0^\infty\int_0^y\dx\  y^p\exp(Ry)\frac{f_I(y)}{1+\t}\dy\nonumber\\
	&=\int_0^\infty y^{p+1}\exp(Ry)\frac{f_I(y)}{1+\t}\dy\nonumber,
	\end{align}
	which is $\O(1)$ by assumption (A\ref{assump:x:exp}). Similar computations can be easily done to get also $\int_0^\infty\aux{\GRbar{p}}(x)\dx=\O(1)$. Thus, again from \eqref{eq:delta:aux:hat:G},
	\begin{equation}\label{eq:delta:aux:int:O:theta}
	\int_0^\infty\vert\aux{\widehat\delta}(x)\vert\dx=\O(\t).
	\end{equation}
	Then, \eqref{eq:Delta:delta:aux:estim}, \eqref{eq:delta:aux:O:theta} and \eqref{eq:delta:aux:int:O:theta} give
	\begin{equation*}
	\aux{\Delta}(u)=\O\left(\vert\aux{\widehat\delta}(u)\vert+\int_0^\infty\vert\aux{\widehat\delta}(x)\vert\dx\right)\exp(-Ru)=\O(\t)\exp(-Ru).
	\end{equation*}
	Therefore,
	\begin{align*}
	\psi(u)-\exp(-Ru)&=\aux{\Delta}+(\psiaux(u)-\exp(-Ru))\\
	&=\O(\t)\exp(-Ru)-\t\exp(-Ru)\\
	&=\O(\t)\exp(-Ru).
	\end{align*}
	We have proved Lemma \ref{lem:psi:exp:err}.
\end{proof}

\subsection{Results for De Vylder type approximations}
We now state and prove key results that are specific to De Vylder type approximations.\\
The following lemma gives necessary and sufficient conditions about the coefficients of any De Vylder type approximation. These conditions will be crucial for the proofs of our main results. Notice that, in general, they do not imply fully explicit expressions of all the approximating parameters (which are usually derived in the literature for particular approximations like De Vylder's original one, mainly for numerical purposes).

\begin{lemma}\label{lem:U:Util}
Let $k\geq 2$. Suppose that both $Z_1$ and $\Ztil_1$ have finite first $k$ moments. The following two assertions are equivalent:
\begin{enumerate}[(i)]
\item\label{item:U:Util} $\E[U_t^j(u)]=\E[\tilde U_t^j(u)]$, for all $j=1\dots k$ and for all $t\geq 0$ and $u\geq 0$.
\item\label{item:mtil:ttil:ltil}
$$\left\{
\begin{aligned}
&\mtil_j=\frac{\mtil_2}{m_2}m_j\textrm{, for all }j=2\dots k,\\
&\ttil=\frac{m_1\mtil_2}{m_2\mtil_1}\t \textrm{, and }\\
&\ltil=\frac{m_2}{\mtil_2}\l.
\end{aligned}
\right.$$
\end{enumerate}
\end{lemma}

\begin{remark}\label{rem:DV:param}
Condition \eqref{item:mtil:ttil:ltil} of Lemma \ref{lem:U:Util} implies that the $\mtil_j$'s do not depend on $\t$, and that $\ttil$ linearly depends on $\t$, so that $\ttil=\O(\t)$. 
\end{remark}

\begin{proof}
Let $M(.,t,u)$ be the moment generating function of $U_t(u)$:
$$M(s,t,u)=\E[\exp(sU_t(u))].$$
Condition \eqref{item:U:Util} of Lemma \ref{lem:U:Util} is equivalent to
\begin{equation}\label{eq:M:Mtil}
\frac{\partial^j M}{\partial s^j}(0,t,u)=\frac{\partial^j \Mtil}{\partial s^j}(0,t,u),
\end{equation}
for all $j=1\dots k$ and for all $t\geq 0$ and $u\geq 0$.\\
We have
$$M(s,t,u)=\exp(L(s,t,u)),$$
where $$L(s,t,u):=su+sct+\l t M_Z(-s)-\l t$$
and $M_Z(-s)=\E[\exp(-sZ_1)]$. Similar quantities are defined for the approximation process $(\Util_t)$, and one has
$$\Mtil(s,t,u)=\exp(su+s\ctil t+\ltil t M_{\Ztil}(-s)-\ltil t)=\exp(\tilde L(s,t,u)).$$
First, we have
\begin{align*}
\frac{\partial M}{\partial s}(0,t,u)&=\frac{\partial L}{\partial s}(0,t,u)M(0,t,u)\\
&=u+ct-\l t m_1\\
&=u+\l m_1\t t,
\end{align*}
so that identity \eqref{eq:M:Mtil} with $j=1$ writes
\begin{equation}\label{eq:lmt:lmttil}
\l m_1 \t=\ltil \mtil_1 \ttil.
\end{equation}
Next, for $j>1$, we have (by Fa\`a di Bruno's formula)
$$\frac{\partial^j M}{\partial s^j}(s,t,u)=\left(\frac{\partial^j L}{\partial s^j}+P(\frac{\partial^{j-1} L}{\partial s^{j-1}},\dots,\frac{\partial L}{\partial s})\right)(s,t,u)M(s,t,u),$$
where $P(.,\dots,.)$ is a polynomial function (with universal constant coefficients).\\
Therefore, by induction on $j$, it is clear that identity \eqref{eq:M:Mtil} with $j=2\dots k$ is equivalent to
$$\frac{\partial^j L}{\partial s^j}(0,t,u)=\frac{\partial^j \tilde L}{\partial s^j}(0,t,u),$$
that is 
\begin{equation}\label{eq:lmk:lmktil}
\l m_j=\ltil \mtil_j, \textrm{, for all }j=2\dots k.
\end{equation}
In conclusion, \eqref{eq:M:Mtil} is equivalent to system \eqref{eq:lmt:lmttil}-\eqref{eq:lmk:lmktil}, which  is equivalent to condition \eqref{item:mtil:ttil:ltil} of Lemma \ref{lem:U:Util}.
\end{proof}

The following lemma shows that De Vylder type approximations are actually approximations of the adjustment coefficient, of order $\t^k$. It generalizes a similar result stated by \cite{gran:00} for De Vylder's original approximation with $k=3$ (where the author yet used the explicit expressions of the parameters available for this particular approximation).

\begin{lemma}\label{lem:R:Rtil}
Let $k\geq 2$. Suppose that each of models \eqref{eq:reserve}
and \eqref{eq:reserve:tilde} satisfies assumptions (A\ref{assump:CL}) and (A\ref{assump:x:exp}).\\
If $\E[U_t^j(u)]=\E[\tilde U_t^j(u)]$, for all $j=1\dots k$ and for all $t\geq 0$ and $u\geq 0$, then
\begin{equation*}
R-\Rtil=\O(\t^k).
\end{equation*}
\end{lemma}

\begin{proof}
From \eqref{eq:def:R}, we have
$$\sum_{j=0}^{k-1}\frac{R^j}{j!}\int_0^\infty x^jf_I(x)\dx+\O(R^k)=1+\t,$$
that is
$$\sum_{j=1}^{k-1}\frac{m_{j+1}}{m_1(j+1)!}R^j+\O(R^k)=\t.$$
Notice that, in particular, we get (see also \cite{gran:00})
\begin{equation}\label{eq:R:O:theta}
R=\frac{2m_1}{m_2}\t+\O(\t^2)=\O(\t).
\end{equation}
Similarly,
$$\sum_{j=1}^{k-1}\frac{\mtil_{j+1}}{\mtil_1(j+1)!}\Rtil^j+\O(\Rtil^k)=\ttil,$$
that is, by virtue of Lemma \ref{lem:U:Util},
$$\sum_{j=1}^{k-1}\frac{m_{j+1}}{m_1(j+1)!}\Rtil^j+\O(\Rtil^k)=\t.$$
Therefore,
$$\sum_{j=1}^{k-1}\frac{m_{j+1}}{m_1(j+1)!}(R^j-\Rtil^j)=\O(\t^k).$$
Since $R^j-\Rtil^j=(R-\Rtil)\sum_{i=0}^{j-1}R^i\Rtil^{j-1-i}$, we get
$$(R-\Rtil)\left(\frac{m_{2}}{2m_1}+\O(\t)\right)=\O(\t^k),$$
which gives
$$R-\Rtil=\O(\t^k).$$
\end{proof}

\section{Approximation error for the ruin probability}\label{sec:err:proba}

Hereafter, the process $U_t$, defined by \eqref{eq:reserve}, is the original risk reserve process with associated ruin probability $\psi(u)$ to be approximated. The superscript "DV" denotes a De Vylder type approximation of order $k\geq 2$, meaning that it is obtained by matching the first $k$ moments of $U_t$ and $\DV{U}_t$.\\
De Vylder's adjustment coefficient, denoted by $\DV{R}$, is just defined by
\begin{equation*}
\int_0^\infty\exp(\DV{R}x)\frac{\DV{f_I}(x)}{1+\DV{\t}}\dx=1.
\end{equation*}
We recall that we do not assume any particular or explicit expression for the approximation parameters $F^{DV}$, $\DV{R}$ and $\psiDV$. Whereas De Vylder's original approximation is exponential and of order $k=3$, our De Vylder type approximation allows for any approximating distribution and for any order $k\geq 2$. Theorem \ref{thm:DV:err:proba} gives an estimate of the approximation error for the ruin probability.

\subsection{Main result}

\begin{theorem}\label{thm:DV:err:proba}
	Suppose that both reserve models $(U_t)$ and $(\DV{U_t})$ satisfy assumptions (A\ref{assump:CL}), (A\ref{assump:x:exp}) and (A\ref{assump:hazard}).\\
Let $k\geq 2$ and $\psiDV(u)$ the $k^\textrm{th}$ order De Vylder type approximation of $\psi(u)$.\\
Then, 
\begin{align}
\psi(u)-\psiDV(u)&=\O\left(\t+\t^ku\right)\exp\left(-(R\wedge \DV{R})u\right)\label{eq:DV:err}\\
&=\O\left(\t+\t^ku\right)\exp\left(-\Big\vert\frac{2m_1}{m_2}\t+\O(\t^2)\Big\vert u\right).\label{eq:DV:err:bis}
\end{align}
For the relative error,
\begin{equation}\label{eq:DV:rel:error}
\frac{\psi(u)-\psiDV(u)}{\psi(u)}=\O\left(\t+\t^ku\right)\exp\left(\O(\t^ku)\indic_{R>\DV{R}}\right).
\end{equation}
\end{theorem}

Before giving the proof of Theorem \ref{thm:DV:err:proba}, let us make some comments and numerical illustrations.\\

As one can see from the proof of Theorem \ref{thm:DV:err:proba}, the $\O(\t)$ term of the error comes from a Lundberg type approximation as given by Lemma \ref{lem:psi:exp:err} (with the exact adjustment coefficient $R$), whereas the $\O(\t^ku)$ term comes from the extra approximation of $R$ by $\DV{R}$, whose accuracy is  $\O(\t^k)$ as given by Lemma \ref{lem:R:Rtil}.\\

In practice, the safety loading coefficient $\t$ is small (and usually smaller than unity). For small $u$, the bound in \eqref{eq:DV:rel:error} becomes $\O(\t)$: this is optimal by considering the relative error at $u=0$, equal to $\vert\tDV-\t\vert(1+\tDV)^{-1}$.\\

For reasonable $u$ (essentially, for $u=\O(\t^{-k+1})$), the second part of the relative error ($\t^ku\exp\left(\O(\t^ku)\right)$ at most) is still small for small $\t$, but becomes non-negligible for larger $u$ (which was also heuristically pointed out by \cite{de:vyld:78}).\\

Figure \ref{fig:DV:err:proba:u:mix:exp} illustrates the above comments. We have computed the (exact) relative error of De Vylder's original approximation for exponentially mixed claims. While Figure \eqref{fig:DV:err:proba:u:mix:exp:small:err}, with $\t=5\%$, shows a $2\%$ relative error when the exact ruin probability $\psi(u)$ reaches $0.5\%$ (the value of Solvency II threshold), Figure \eqref{fig:DV:err:proba:u:mix:exp:big:err} shows that, already with $\t=20\%$ and as soon as $\psi(u)$ goes below $0.5\%$, the corresponding De Vylder relative error exceeds $21\%$! Therefore, this is a practical situation where one observes a blow-up of the error.\\

\begin{figure}[!htbp]
	\centering
	\begin{subfigure}{0.8\textwidth}
		\centering
		\includegraphics[width=\textwidth]{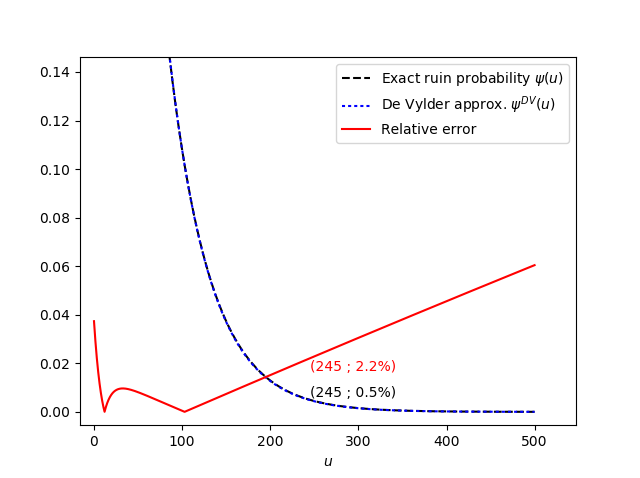}
		\caption{With $\t=5\%$. At $u=245$, where the exact ruin probability reaches $0.5\%$ (Solvency II threshold), De Vylder's relative error is about $2\%$.}
		\label{fig:DV:err:proba:u:mix:exp:small:err}
	\end{subfigure}
	\begin{subfigure}{0.8\textwidth}
	\includegraphics[width=\textwidth]{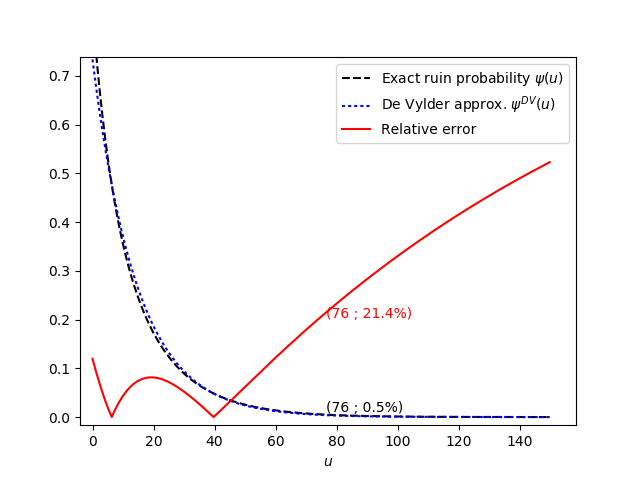}
	\caption{With $\t=20\%$. At $u=76$, where the exact ruin probability reaches $0.5\%$ (Solvency II threshold), De Vylder's relative error goes above $21\%$!}
	\label{fig:DV:err:proba:u:mix:exp:big:err}
	\end{subfigure}
	\caption{De Vylder approximation of the ruin probability, as a function of $u$, with two different $\t$'s. The claims are exponentially mixed with density $a\beta_1\exp(-\beta_1x)+(1-a)\beta_2\exp(-\beta_2x)$, with $a=0.01$, $\beta_1=0.1$, $\beta_2=0.6$.}
	\label{fig:DV:err:proba:u:mix:exp}
\end{figure}

The bound obtained in \eqref{eq:DV:rel:error} for the relative error shows either a linear or an exponential blow-up with respect to $u$, depending on whether $R\leq \RDV$ or the opposite. This is confirmed by the numerical examples illustrated by Figures \ref{fig:DV:err:u} and \ref{fig:DV:err:theta}, where we have computed the relative error of De Vylder's original approximation in different situations.\\

\begin{figure}[!htbp]
	\centering
	\begin{subfigure}{0.8\textwidth}
		\centering
		\includegraphics[width=\textwidth]{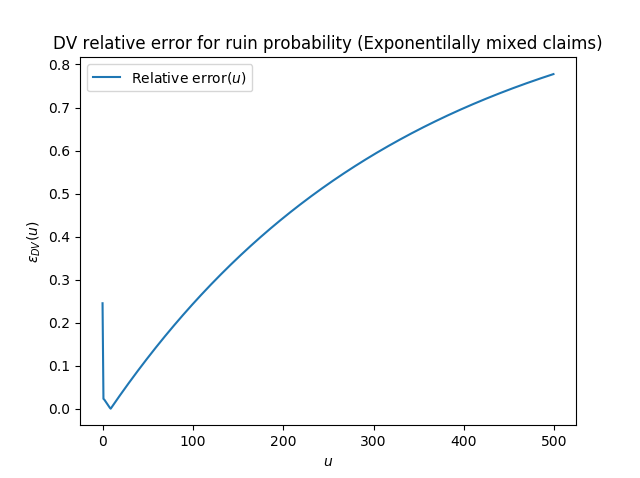}
		\caption{With exponentially mixed claims ($R<\RDV$)}
		\label{fig:DV:err:u:mix:exp}
	\end{subfigure}
	\begin{subfigure}{0.8\textwidth}
		\centering
		\includegraphics[width=\textwidth]{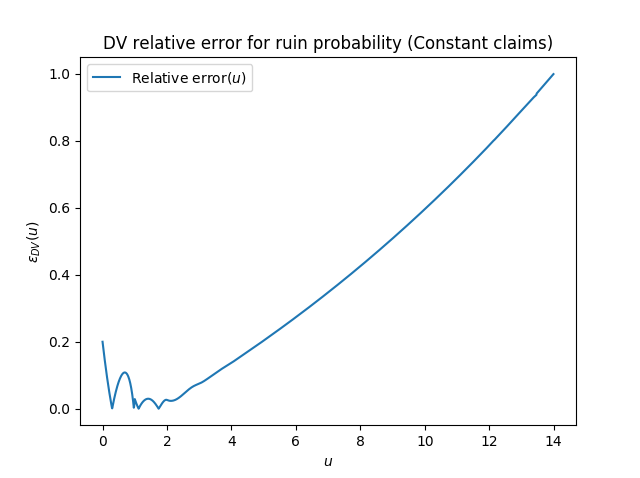}
		\caption{With deterministic claims ($R>\RDV$)}
		\label{fig:DV:err:u:determ}
	\end{subfigure}
	\caption{Two examples of De Vylder's approximation relative error for the ruin probability, as a function of $u$ (with $\t=1$). In Figure \eqref{fig:DV:err:u:mix:exp}, the claims density is $a\beta_1\exp(-\beta_1x)+(1-a)\beta_2\exp(-\beta_2x)$, with $a=0.0584$, $\beta_1=0.359$, $\beta_2=7.5088$ (as in \cite{cize:hard:wero:11}). In Figure \eqref{fig:DV:err:u:determ}, the claims are constant equal to 1 (see \cite{shiu:88} for explicit expressions).}
	\label{fig:DV:err:u}
\end{figure}

\begin{figure}[!htbp]
	\centering
	\begin{subfigure}{0.8\textwidth}
		\centering
		\includegraphics[width=\textwidth]{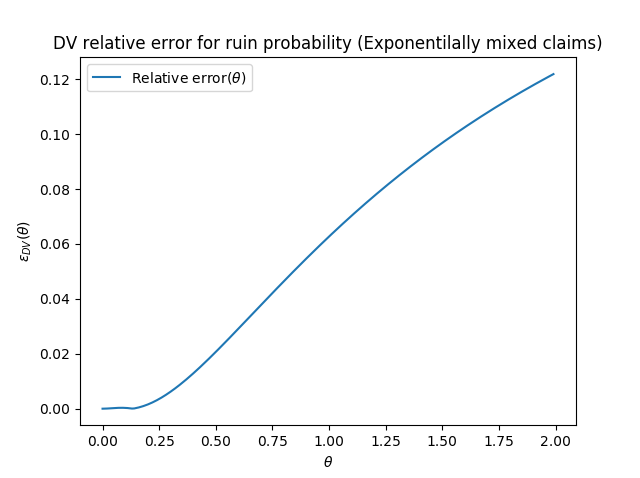}
		\caption{With exponentially mixed claims ($R<\RDV$)}
		\label{fig:DV:err:theta:mix:exp}
	\end{subfigure}
	\begin{subfigure}{0.8\textwidth}
		\centering
		\includegraphics[width=\textwidth]{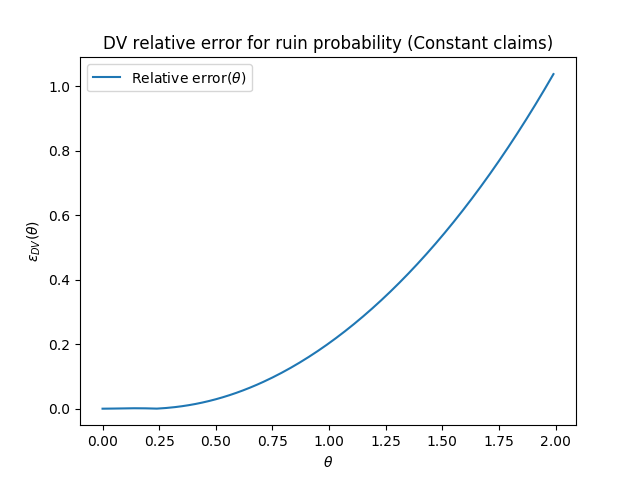}
		\caption{With deterministic claims ($R>\RDV$)}
		\label{fig:DV:err:theta:determ}
	\end{subfigure}
	\caption{Two examples of De Vylder's approximation relative error for the ruin probability, as a function of $\theta$. In Figure \eqref{fig:DV:err:theta:mix:exp}, $u=30$ and in \eqref{fig:DV:err:theta:determ}, $u=5$. The values of the other parameters are the same as in Figure \ref{fig:DV:err:u}.}
	\label{fig:DV:err:theta}
\end{figure}

Besides, for fixed $\t<1$ and $u>0$, the term $\t^ku$ is smaller for higher approximation order $k$, so that the bound  \eqref{eq:DV:rel:error} for the relative error becomes smaller. This explains the numerical results obtained by \cite{burn:mist:wero:05} for a Gamma type approximation with $k=4$, where an improvement of the relative error was numerically observed, in comparison with De Vylder's original approximation with $k=3$.

\subsection{Proof of Theorem \ref{thm:DV:err:proba}}\label{sec:proof:thm:DV:err:proba}


\begin{proof}[Proof of Theorem \ref{thm:DV:err:proba}]

By Lemma \ref{lem:psi:exp:err}, we have
\begin{equation*}
\psi(u)-\exp(-Ru)=\O(\t)\exp(-Ru)
\end{equation*}
and
\begin{align*}
\psiDV(u)-\exp(-\DV{R}u)&=\O(\DV{\t})\exp(-\DV{R}u)\\
&=\O(\t)\exp(-\DV{R}u),
\end{align*}
where we have used Lemma \ref{lem:U:Util} and Remark \ref{rem:DV:param}. Then,
\begin{align*}
\psi(u)-\psiDV(u)&=(\psi(u)-\exp(-Ru))-(\psiDV(u)-\exp(-\DV{R}u))\\
&+\exp(-Ru)-\exp(-\DV{R}u)\\
&=\O(\t)\exp(-Ru)+\O(\t)\exp(-\DV{R}u)\\
&+\O(R-\DV{R})u\exp\left(-(R\wedge \DV{R})u\right)\\
&=\O(\t+\t^ku)\exp\left(-(R\wedge \DV{R})u\right),
\end{align*}
where we have used Lemma \ref{lem:R:Rtil}. We have proved \eqref{eq:DV:err}.\\
By \eqref{eq:R:O:theta}, we have
\begin{align*}
R&=\frac{2m_1}{m_2}\t+\O(\t^2);\\
\RDV&=\frac{2\mDV_1}{\mDV_2}\tDV+\O((\tDV)^2).
\end{align*}
Now, from Lemma \ref{lem:U:Util}, we have
\begin{equation*}
\frac{m_1}{m_2}\t=\frac{\mDV_1}{\mDV_2}\tDV.
\end{equation*}
Thus
\begin{equation*}
R\wedge\RDV=\frac{2m_1}{m_2}\t+\O(\t^2),
\end{equation*}
and identity \eqref{eq:DV:err:bis} follows from \eqref{eq:DV:err}.\\
From \eqref{eq:DV:err}, and using Lemma \ref{lem:psi:exp:err} (which says that $\psi(u)=(1+\O(\t))\exp(-Ru)$), estimate \eqref{eq:DV:rel:error} for the relative error is straightforward. The proof of Theorem \ref{thm:DV:err:proba} is complete.
\end{proof}

\subsection{Comparison with other approximations}\label{sec:other:approx}
We can take advantage of 
Lemma \ref{lem:psi:exp:err} in order to derive error estimates for known exponential type approximations.\\
First, it turns out that the classical Cramer-Lundberg approximation $\psiCL(u)$ is of the form $(1+\O(\t))\exp(-Ru)$. Indeed, $\psiCL(u)=\alpha\exp(-Ru)$, with
\begin{align}
\alpha&=\frac{\t}{R\int_0^\infty x\exp(Rx)f_I(x)\dx}\nonumber\\
&=\frac{\t}{R\left(\int_0^\infty xf_I(x)\dx+\O(R)\right)}\nonumber\\
&=\frac{\t}{R\left(\frac{m_2}{2m_1}+\O(\t)\right)}\nonumber\\
&=\frac{\t}{\left(\frac{2m_1}{m_2}\t+\O(\t^2)\right)\left(\frac{m_2}{2m_1}+\O(\t)\right)}=1+\O(\t),\nonumber
\end{align}
where we have used \eqref{eq:R:O:theta}.\\
Therefore, and thanks to Lemma \ref{lem:psi:exp:err}, we obtain an estimate for the Cramer-Lundberg approximation:
\begin{equation}\label{eq:CL:err:estim}
\psi(u)-\psiCL(u)=\O(\t)\exp(-Ru),
\end{equation}
then with an $\O(\t)$ relative error. It though requires exact knowledge of the adjustment coefficient $R$.\\
We point out that estimate \eqref{eq:CL:err:estim} is optimal, with respect to our parameter of interest $\t$, for small $u$ (one can easily check that $\psi(0)-\psiCL(0)=C\t$, with a positive constant $C$). However, we cannot claim that it is optimal with respect to large $u$ (for the case of bounded claims, which is not ours, we refer the interested reader to \cite{ekhe:silv:11} and \cite{silv:mart:14}, where the authors state a relative error that is a decaying exponential of $u$ that depends on the trucation bound of the claims).\\
We can go further and deduce error bounds for all approximations of the form $$(1+\O(\t))\exp(-\widehat R u),$$
where $\widehat R$ is an approximation of $R$. These include, in addition to Cramer-Lundberg's one, De Vylder's original $\psiDV$ with $k=3$, Lundberg's $\psi^{\textrm L}$, R\'enyi's $\psi^{\textrm R}$, the diffusion $\psi^{\textrm D}$ and the exponential $\psi^{\textrm E}$ approximations (see \cite{gran:00} for more details).\\
For such approximations (and still using Lemma \ref{lem:psi:exp:err}), the error is
\begin{equation}\label{eq:exp:approx:err:estim}
\O\Big(\t+\varepsilon_Ru\Big)\exp\left(-(R\wedge\widehat R)u\right),
\end{equation}
where $\varepsilon_R:=\vert R-\widehat R\vert$.\\
It turns out that $\psi^{\textrm L}(u)$, $\psi^{\textrm R}(u)$, $\psi^{\textrm D}(u)$ and $\psi^{\textrm E}(u)$ are all of the form
$$\left(1+\O(\t)\right)\exp\left(-\frac{2m_1}{m_2}\t(1+\widehat\O(\t))u\right).$$ 
Now, remember from \eqref{eq:R:O:theta} that $\frac{2m_1}{m_2}\t$ is just the first order approximation of $R$ with respect to $\t$, which means that the $\varepsilon_R$ corresponding to these three approximations is at least $\O(\t^2)$, and the total error is
\begin{equation}\label{eq:R:D:L:err:estim}
\O\left(\t+\t^2u\right)\exp\left(-\Big\vert\frac{2m_1}{m_2}\t+\O(\t^2)\Big\vert u\right).
\end{equation}
In contrast, original De Vylder's $\varepsilon_R$ is $\O(\t^3)$, which explains the better observed accuracy of the latter compared to the three former approximations.\\
Let us point out here that, by \eqref{eq:R:D:L:err:estim}, we have obtained an improved, pointwise, estimate for Renyi's approximation error in comparison with the existing one of \cite[Lemma 2.2, \pno 177]{kala:97}, where it is only stated that the supremum norm (with respect to $u$) of the error is $\O(\t)$.

\section{Approximation error for the moments of the time of ruin}\label{sec:err:time}

Let us denote by $t_j$ ($j=1,2,\dots$) the $j^\textrm{th}$ moment of the time of ruin given that ruin occurs:
\begin{equation*}
t_j(u):=\E[T^j\ \big\vert\ T<\infty].
\end{equation*}

\subsection{Main result}

\begin{theorem}\label{thm:DV:err:time}
	Suppose that both reserve models $(U_t)$ and $(\DV{U_t})$ satisfy assumptions (A\ref{assump:CL}), (A\ref{assump:x:exp}) and (A\ref{assump:hazard}).\\
	Let $k\geq 2$, $j\in\{1,2,\dots\}$, and $\ttDV_j(u)$ the $k^\textrm{th}$ order De Vylder type approximation of $t_j(u)$.\\
	Then, 
	\begin{equation*}
	t_j(u)-\ttDV_j(u)=\O\left(\sum_{i=0}^{j} u^{j-i}\t^{1-i-j}\right).
	\end{equation*}
	For the relative error,
	\begin{align}
	\frac{t_j(u)-\ttDV_j(u)}{t_j(u)}&=\O\left(\frac{\sum_{i=0}^{j} u^{i}\t^{i}}{\sum_{i=0}^{j-1}u^{i+1}\t^{i}+\O(1)}\right).\label{eq:rel:err:time}
	\end{align}
\end{theorem}

The bound on the relative error \eqref{eq:rel:err:time} in Theorem \ref{thm:DV:err:time} is at least $\O(1)$ for small $u$, and one can check that this rate is optimal for $u=0$ by easy explicit computations. On the other hand, when $u\to\infty$, the bound in \eqref{eq:rel:err:time} becomes equal to $\O(\t)$.\\
These two different behaviours of the error, for small $u$ and for large $u$, were already numerically pointed out by \cite{dick:wong:04} for De Vylder approximation of the moments of the time of ruin. They are confirmed by our numerical illustrations in Figures \ref{fig:DV:approx:time} and \ref{fig:DV:err:time}.

\begin{figure}[!htbp]
		\centering
		\includegraphics[width=0.8\textwidth]{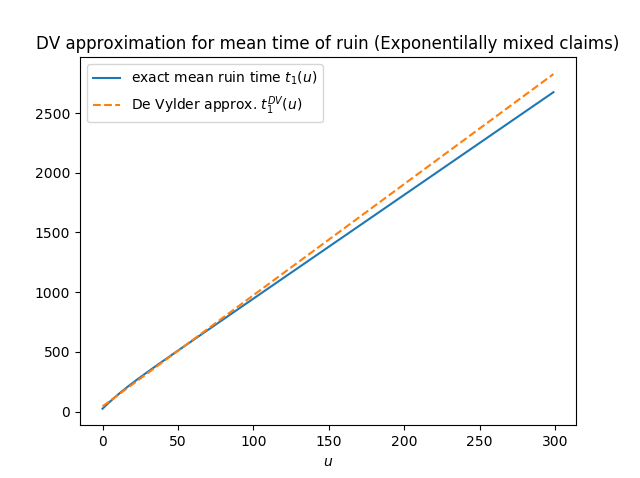}
		\caption{Exact value and De Vylder approximation for the average time of ruin, as functions of $u$ (with $\t=10\%$). The claims density is $a\beta_1\exp(-\beta_1x)+(1-a)\beta_2\exp(-\beta_2x)$, with $a=0.01$, $\beta_1=0.1$, $\beta_2=0.6$ (see \cite{lin:will:00} for explicit expressions).}
	\label{fig:DV:approx:time}
\end{figure}

\begin{figure}[!htbp]
	\centering
	\begin{subfigure}{0.8\textwidth}
		\centering
		\includegraphics[width=\textwidth]{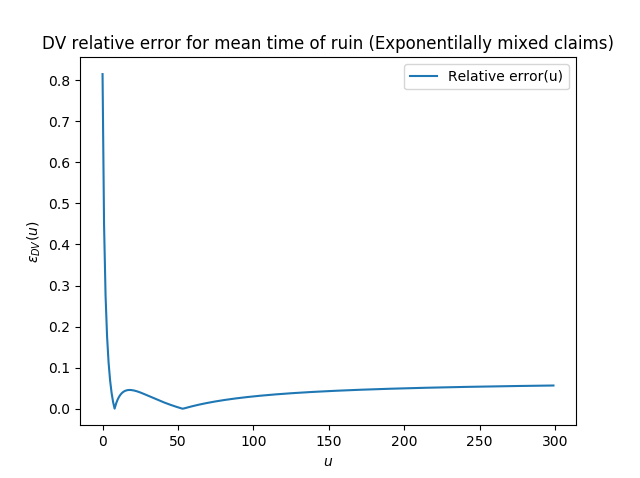}
		\caption{Relative error, as a function of $u$ (with $\t=10\%$)}
	\end{subfigure}
	\begin{subfigure}{0.8\textwidth}
		\centering
		\includegraphics[width=\textwidth]{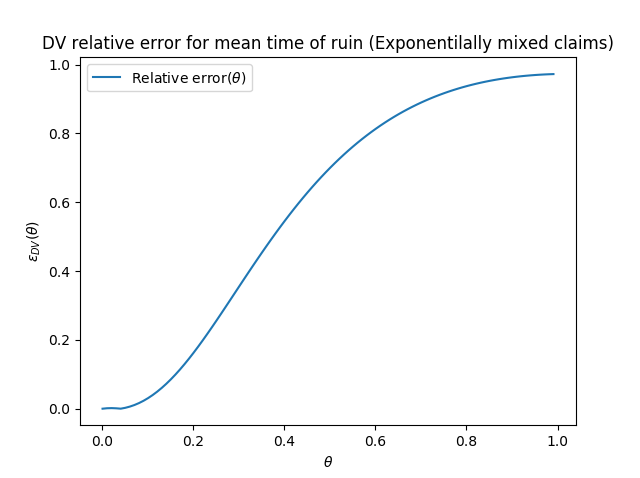}
		\caption{Relative error, as a function of $\t$ (with $u=100$)}
	\end{subfigure}
	\caption{De Vylder's approximation relative error for the average time of ruin. The claims density is the same as in Figure \ref{fig:DV:approx:time}.}
	\label{fig:DV:err:time}
\end{figure}

\subsection{Proof of Theorem \ref{thm:DV:err:time}}

\begin{proof}[Proof of Theorem \ref{thm:DV:err:time}]
	Let us begin with the first moment. From \cite[Corollary 6.1]{lin:will:00}, we have
	\begin{equation*}
	t_1(u)=\frac{1}{\lambda m_1\t}\left(\int_0^u\frac{\psi(u-x)\psi(x)}{\psi(u)}\dx+\int_u^\infty\frac{\psi(x)}{\psi(u)}\dx-\frac{m_2}{2m_1\t}\right).
	\end{equation*}
	From Lemma \ref{lem:psi:exp:err}, we have $\psi(u)=(1+\O(\t))\exp(-Ru)$. Then,
	\begin{align}
	t_1(u)&=\frac{1}{\lambda m_1\t}\left(\int_0^u(1+\O(\t))\dx+\int_u^\infty(1+\O(\t))\exp(-R(x-u))\dx-\frac{m_2}{2m_1\t}\right)\nonumber\\
	&=\frac{1}{\lambda m_1\t}\left(u+R^{-1}-\frac{m_2}{2m_1\t}\right)+\O(1)\left(u+R^{-1}\right).\label{eq:e:1:u:R}
	\end{align}
	Similarly,
	\begin{equation*}
	\ttDV_1(u)=\frac{1}{\lDV \mDV_1\tDV}\left(u+(\RDV)^{-1}-\frac{\mDV_2}{2\mDV_1\tDV}\right)+\O(1)\left(u+(\RDV)^{-1}\right),
	\end{equation*}
	which equals, using Lemma \ref{lem:U:Util},
	\begin{equation*}
	\ttDV_1(u)=\frac{1}{\lambda m_1\t}\left(u+(\RDV)^{-1}-\frac{m_2}{2m_1\t}\right)+\O(1)\left(u+(\RDV)^{-1}\right).
	\end{equation*}
	Thus,
	\begin{equation*}
	t_1(u)-\ttDV_1(u)=\frac{1}{\lambda m_1\t}\left(R^{-1}-(\RDV)^{-1}\right)+\O(1)\left(u+R^{-1}+(\RDV)^{-1}\right).
	\end{equation*}
	Now, by \eqref{eq:R:O:theta}, we have
	\begin{equation}\label{eq:R:-i}
	\frac{1}{R^i}=\O(\t^{-i})
	\end{equation}
	and, by Lemma \ref{lem:R:Rtil},
	\begin{equation}\label{eq:R:RDV:-i}
	\frac{1}{R^i}-\frac{1}{(\RDV)^i}=\frac{(\RDV)^i-R^i}{(R\RDV)^i}=\O(\frac{\t^k\t^{i-1}}{\t^{2i}})=\O(\t^{k-1-i}).
	\end{equation}
	Therefore,
	\begin{equation*}
	t_1(u)-\ttDV_1(u)=\O(\t^{k-3}+u+\t^{-1})=\O(u+\t^{-1}).
	\end{equation*}
	Besides, by \eqref{eq:R:O:theta}, we get that $R^{-1}-\frac{m_2}{2m_1\t}=\O(1)$, which, plugged in \eqref{eq:e:1:u:R}, gives
	$$t_1(u)\geq C(u+\O(1))\t^{-1},$$
	where $C$ is a positive constant. Therefore
	\begin{equation*}
	\frac{t_1(u)-\ttDV_1(u)}{t_1(u)}=\O\left(\frac{\t u+1}{u+\O(1)}\right).
	\end{equation*}
	We now investigate the approximation of higher moments. From \cite[Theorem 6.3]{lin:will:00}, we have, for $j=1,2,\dots$,
	\begin{equation*}
	t_j(u)=\frac{j}{\lambda m_1\t}\left(\int_0^u\frac{\psi(u-x)\psi(x)}{\psi(u)}t_{j-1}(x)\dx+\int_u^\infty\frac{\psi(x)}{\psi(u)}t_{j-1}(x)\dx-\int_0^\infty\psi(x)t_{j-1}(x)\dx\right),
	\end{equation*}
	Still by Lemma \ref{lem:psi:exp:err}, we get
	\begin{align*}
	t_j(u)&=\frac{j}{\lambda m_1\t}\left(\int_0^u t_{j-1}(x)\dx+\int_0^\infty \exp(-Rx)\left(t_{j-1}(u+x)-t_{j-1}(x)\right)\dx\right)\\
	&+\O(1)\left(\int_0^u t_{j-1}(x)\dx+\int_0^\infty \exp(-Rx)\left(t_{j-1}(u+x)+t_{j-1}(x)\right)\dx\right).
	\end{align*}
	By induction on $j=1,2,\dots$, it is not difficult to deduce that
	\begin{align*}
	t_j(u)&=\frac{(1+\O(\t))}{(\lambda m_1\t)^j}\sum_{i=0}^{j-1}C_{i,j} u^{j-i}R^{-i}+\frac{1}{(\lambda m_1\t)^{j-1}}\O(R^{-j}),
	\end{align*}
	where $(C_{i,j})$ are positive universal constants. As for the first moment, we then have
	\begin{align*}
	t_j(u)-\ttDV_j(u)&=\frac{1}{(\lambda m_1\t)^j}\sum_{i=0}^{j-1}C_{i,j} u^{j-i}(R^{-i}-(\RDV)^{-i})\\
	&+\frac{1}{(\lambda m_1\t)^{j-1}}\O\left(\sum_{i=0}^{j-1}u^{j-i}\t^{-i}+\t^{-j}\right)\\
	&=\O\left(\sum_{i=0}^{j-1}u^{j-i}\t^{-j+k-1-i}+\sum_{i=0}^{j}u^{j-i}\t^{-j+1-i}\right)\\
	&=\O\left(\sum_{i=0}^{j} u^{j-i}\t^{1-i-j}\right),
	\end{align*}
	and
	\begin{align*}
	\frac{t_j(u)-\ttDV_j(u)}{t_j(u)}&=\O\left(\frac{\sum_{i=0}^{j} u^{j-i}\t^{1-i-j}}{\sum_{i=0}^{j-1}u^{j-i}\t^{-j-i}+\O(\t^{-2j+1})}\right)\\
	&=\O\left(\frac{\sum_{i=0}^{j} u^{i}\t^{i}}{\sum_{i=0}^{j-1}u^{i+1}\t^{i}+\O(1)}\right).
	\end{align*}
\end{proof}

\section{Approximation error for the moments of the deficit at ruin}\label{sec:err:deficit}

The deficit at ruin is defined by $\vert U_T\vert$. Let us denote by $d_j$ ($j=1,2,\dots$) its $j^\textrm{th}$ moment given that ruin occurs:
\begin{equation*}
d_j(u):=\E[\vert U_T\vert^j\ \big\vert\ T<\infty].
\end{equation*}

\subsection{Main result}

It turns out that, while De Vylder's approximation fits the first moments of the surplus process, it fails to accurately approximate the moments of the surplus (deficit) at ruin!
\begin{theorem}\label{thm:DV:err:deficit}
	Suppose that both reserve models $(U_t)$ and $(\DV{U_t})$ satisfy assumptions (A\ref{assump:CL}), (A\ref{assump:x:exp}) and (A\ref{assump:hazard}).\\
	Let $k\geq 2$, $j\in\{1,\dots,k-1\}$, and $\dDV_j(u)$ the $k^\textrm{th}$ order De Vylder type approximation of $d_j(u)$.\\
	Then, 
	\begin{equation*}
	d_j(u)-\dDV_j(u)
	=\O(1).
	\end{equation*}
	For the relative error,
	\begin{equation}
	\frac{d_j(u)-\dDV_j(u)}{d_j(u)}
	=\O(1).
	\label{eq:DV:err:deficit}
	\end{equation}
\end{theorem}

	The relative error estimate \eqref{eq:DV:err:deficit} is optimal at least for small $u$. Indeed, for $u=0$, the density of the deficit at ruin is known to be exactly $f_I$ (see \cite{kaas:goov:dhae:denu:08}), and then $d_j(0)=\frac{m_{j+1}}{(j+1)m_1}$. Hence, by Lemma \ref{lem:U:Util}, $d_j(0)$ cannot "cancel" or be compared with $\dDV_j(0)$, so that the relative error is exactly $\O(1)$ (and not smaller). Estimate \eqref{eq:DV:err:deficit}, for all $u$ and $\t$, is confirmed by numerical experiments illustrated in Figures \ref{fig:DV:approx:deficit} and \ref{fig:DV:err:deficit}.

\begin{figure}[!htbp]
	\centering
	\includegraphics[width=0.8\textwidth]{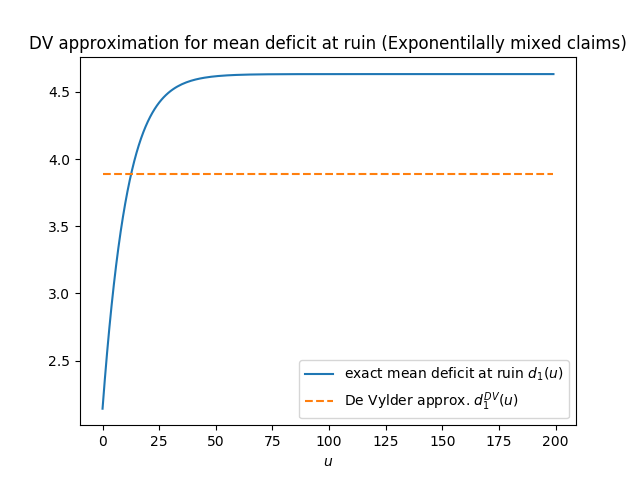}
		\caption{Exact value and De Vylder approximation for the average deficit at ruin, as functions of $u$ (with $\t=10\%$). The claims density is $a\beta_1\exp(-\beta_1x)+(1-a)\beta_2\exp(-\beta_2x)$, with $a=0.01$, $\beta_1=0.1$, $\beta_2=0.6$ (see \cite{lin:will:00} for explicit expressions).}
		\label{fig:DV:approx:deficit}
\end{figure}

\begin{figure}[!htbp]
	\centering
	\begin{subfigure}{0.8\textwidth}
		\centering
		\includegraphics[width=\textwidth]{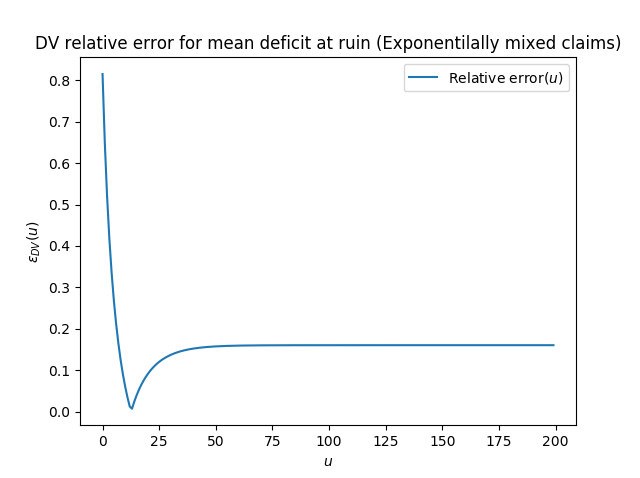}
		\caption{Relative error, as a function of $u$ (with $\t=10\%$)}
	\end{subfigure}
	\begin{subfigure}{0.8\textwidth}
		\centering
		\includegraphics[width=\textwidth]{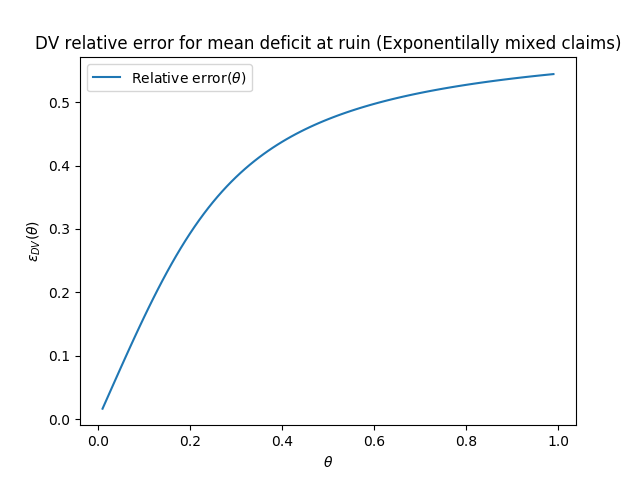}
		\caption{Relative error, as a function of $\t$ (with $u=100$)}
	\end{subfigure}
	\caption{De Vylder's approximation relative error for the average deficit at ruin. The claims density is the same as in Figure \ref{fig:DV:approx:deficit}.}
	\label{fig:DV:err:deficit}
\end{figure}

\subsection{Proof of Theorem \ref{thm:DV:err:deficit}}

\begin{proof}[Proof of Theorem \ref{thm:DV:err:deficit}]
From \cite[Corollary 4.1]{lin:will:00}, we know that
\begin{equation*}
d_j(u)=\frac{m_j}{m_1\t}\frac{\tau_j(u)}{\psi(u)}-\frac{m_{j+1}}{(j+1)m_1\t},
\end{equation*}
where
\begin{equation*}
\tau_j(u)=\frac{jm_1\t}{m_j}\int_u^\infty(x-u)^{j-1}\psi(x)\dx-\sum_{i=0}^{j-2}\binom{j}{i}\frac{m_{j-i}}{m_j}\int_u^\infty(x-u)^i\psi(x)\dx.
\end{equation*}
From Lemma \ref{lem:psi:exp:err}, we have $\psi(u)=(1+\O(\t))\exp(-Ru)$. Then,
\begin{align*}
d_j(u)&=(1+\O(\t))\left(j\int_u^\infty(x-u)^{j-1}\exp(-R(x-u))\dx\right.\\
&-\left.\sum_{i=0}^{j-2}\binom{j}{i}\frac{m_{j-i}}{m_1\t}\int_u^\infty(x-u)^i\exp(-R(x-u))\dx\right)-\frac{m_{j+1}}{(j+1)m_1\t}.
\end{align*}
By the change of variable $y=R(x-u)$, we get
\begin{equation}\label{eq:dj:R}
d_j(u)=(1+\O(\t))\left(\frac{j!}{R^j}-\sum_{i=0}^{j-2}\binom{j}{i}\frac{m_{j-i}}{m_1\t}\frac{i!}{R^{i+1}}\right)-\frac{m_{j+1}}{(j+1)m_1\t}.
\end{equation}
Similarly, we have
\begin{equation}\label{eq:djDV:RDV}
\dDV_j(u)=(1+\O(\tDV))\left(\frac{j!}{(\RDV)^j}-\sum_{i=0}^{j-2}\binom{j}{i}\frac{\mDV_{j-i}}{\mDV_1\tDV}\frac{i!}{(\RDV)^{i+1}}\right)-\frac{\mDV_{j+1}}{(j+1)\mDV_1\tDV}.
\end{equation}
By Lemma \ref{lem:U:Util} and for $i=2,\dots,k$,
$$\frac{\mDV_i}{\mDV_1\tDV}=\frac{\frac{\mDV_2}{m_2}m_i}{\mDV_1\frac{m_1\mDV_2}{m_2\mDV_1}\t}=\frac{m_i}{m_1\t}.$$
Thus, for $j=1,\dots,k-1$, \eqref{eq:dj:R} and \eqref{eq:djDV:RDV} yield
\begin{align}
d_j(u)-\dDV_j(u)&=(1+\O(\t))\left(j!(\frac{1}{R^j}-\frac{1}{(\RDV)^j})-\sum_{i=0}^{j-2}\binom{j}{i}\frac{m_{j-i}}{m_1\t}i!(\frac{1}{R^{i+1}}-\frac{1}{(\RDV)^{i+1}})\right)\nonumber\\
&+\O(\t)\left(\frac{j!}{R^j}-\sum_{i=0}^{j-2}\binom{j}{i}\frac{m_{j-i}}{m_1\t}\frac{i!}{R^{i+1}}\right)\nonumber\\
&:=(1+\O(\t))A+\O(\t)B.\label{eq:dj:djDV}
\end{align}
Using \eqref{eq:R:-i} and \eqref{eq:R:RDV:-i},
\begin{equation*}
A=\O(\t^{k-1-j}).
\end{equation*}
The estimate of the term $B$ is more tricky. We have
\begin{align*}
B&=\frac{j!}{R^j}-\sum_{i=0}^{j-2}\binom{j}{i}\frac{m_{j-i}}{m_1\t}\frac{i!}{R^{i+1}}\\
&=\frac{j!}{R^j\t}\left(\t-\sum_{i=0}^{j-2}\frac{m_{j-i}}{m_1(j-i)!}R^{j-i-1}\right)\\
&=\frac{j!}{R^j\t}\left(\t-\sum_{i=1}^{j-1}\frac{m_{i+1}}{m_1(i+1)!}R^{i}\right)\\
&=\frac{j!}{R^j\t}\left(1+\t-\sum_{i=0}^{j-1}\frac{m_{i+1}}{m_1(i+1)}\frac{R^{i}}{i!}\right)\\
&=\frac{j!}{R^j\t}\left(1+\t-\int_0^\infty\sum_{i=0}^{j-1}\frac{R^{i}}{i!}x^if_I(x)\dx\right).
\end{align*}
Using the following straightforward inequality (that holds for any $z\geq 0$):
\begin{equation}\label{eq:ineq:expo:taylor}
\exp(z)-\frac{z^j}{j!}\exp(z)\leq\sum_{i=0}^{j-1}\frac{z^i}{i!}\leq \exp(z)-\frac{z^j}{j!},
\end{equation}
we get (with $Rx$ playing the role of $z$)
\begin{align*}
B&\geq\frac{j!}{R^j\t}\left(1+\t-\int_0^\infty (\exp(Rx)-\frac{R^j}{j!}x^j)f_I(x)\dx\right)\\
&=\frac{1}{\t}\int_0^\infty x^j f_I(x)\dx
\end{align*}
(using \eqref{eq:def:R}), and
\begin{align*}
B&\leq \frac{j!}{R^j\t}\left(1+\t-\int_0^\infty (\exp(Rx)-\frac{R^j}{j!}x^j\exp(Rx))f_I(x)\dx\right)\\
&=\frac{1}{\t}\int_0^\infty x^j\exp(Rx)f_I(x)\dx.
\end{align*}
By Assumption (A\ref{assump:x:exp}), we obtain
\begin{equation*}
B=\O(\t^{-1}).
\end{equation*}
Back to \eqref{eq:dj:djDV}, we deduce that
\begin{equation*}
d_j(u)-\dDV_j(u)=(1+\O(\t))\O(\t^{k-1-j})+\O(\t)\O(\t^{-1})=\O(1).
\end{equation*}
For a lower bound on $d_j(u)$, we have from \eqref{eq:dj:R} (and like for the term $B$ above)
\begin{align*}
d_j(u)&=(1+\O(\t))\frac{j!}{R^j\t}\left(\t-\sum_{i=1}^{j}\frac{m_{i+1}}{m_1(i+1)!}R^{i}\right)\\
&=(1+\O(\t))\frac{j!}{R^j\t}\left(1+\t-\int_0^\infty\sum_{i=0}^{j}\frac{R^{i}}{i!}x^if_I(x)\dx\right).
\end{align*}
Then, again by \eqref{eq:ineq:expo:taylor}, \eqref{eq:def:R} and \eqref{eq:R:O:theta},
\begin{align*}
d_1(u)&\geq (1+\O(\t))\frac{j!}{R^j\t}\left(1+\t-\int_0^\infty (\exp(Rx)-\frac{R^{j+1}}{(j+1)!}x^{j+1})f_I(x)\dx\right)\\
&=(1+\O(\t))\frac{R}{j\t}\int_0^\infty x^{j+1} f_I(x)\dx\\
&=\frac{(1+\O(\t))}{j}(\frac{2m_1}{m_2}+\O(\t))\int_0^\infty x^{j+1} f_I(x)\dx\\
&=C+\O(\t)
\end{align*}
(with $C$ a positive constant). Thus,
\begin{equation*}
\frac{d_j(u)-\dDV_j(u)}{d_j(u)}=\O(\frac{1}{C+\O(\t)})=\O(1).
\end{equation*}
\end{proof}

\section{Approximation error for the moments of the surplus before ruin}\label{sec:err:surplus}

The surplus before ruin is defined by $U_{T-}$. Let us denote by $s_j$ ($j=1,2,\dots$) its $j^\textrm{th}$ moment given that ruin occurs:
\begin{equation*}
s_j(u):=\E[U_{T-}^j\ \big\vert\ T<\infty].
\end{equation*}

\subsection{Main result}

Like the approximation of the moments of the deficit at ruin (Theorem \ref{thm:DV:err:deficit}), De Vylder type methods also fail to accurately approximate the moments of the surplus before ruin!

\begin{theorem}\label{thm:DV:err:surplus}
	Suppose that both reserve models $(U_t)$ and $(\DV{U_t})$ satisfy assumptions (A\ref{assump:CL}), (A\ref{assump:x:exp}) and (A\ref{assump:hazard}).\\
	Let $k\geq 2$, $j\in\{1,\dots,k-1\}$, and $\sDV_j(u)$ the $k^\textrm{th}$ order De Vylder type approximation of $s_j(u)$.\\
	Then, 
	\begin{equation*}
	s_j(u)-\sDV_j(u)=\O(1).
	\end{equation*}
	For the relative error,
	\begin{equation}
	\frac{s_j(u)-\sDV_j(u)}{s_j(u)}=\O(1).\label{eq:DV:err:surplus}
	\end{equation}
\end{theorem}

Estimate \eqref{eq:DV:err:surplus} is confirmed by Figures \ref{fig:DV:approx:surplus} and \ref{fig:DV:err:surplus}.

\begin{figure}[!htbp]
	\centering
	\includegraphics[width=0.8\textwidth]{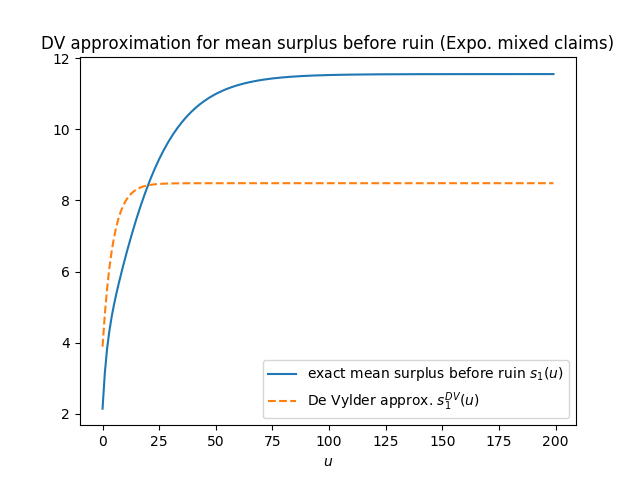}
	\caption{Exact value and De Vylder approximation for the average surplus before ruin, as functions of $u$ (with $\t=10\%$). The claims density is $a\beta_1\exp(-\beta_1x)+(1-a)\beta_2\exp(-\beta_2x)$, with $a=0.01$, $\beta_1=0.1$, $\beta_2=0.6$ (see \cite{lin:will:00} for explicit expressions).}
	\label{fig:DV:approx:surplus}
\end{figure}
\begin{figure}[!htbp]
	\centering
	\begin{subfigure}{0.8\textwidth}
		\includegraphics[width=\textwidth]{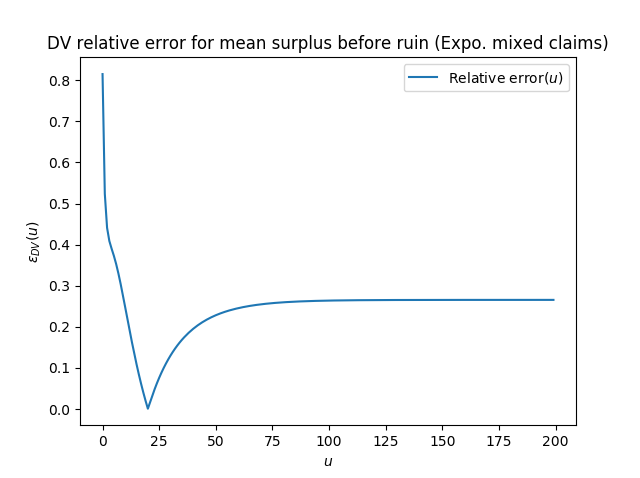}
		\caption{Relative error, as a function of $u$ (with $\t=10\%$)}
	\end{subfigure}
	\begin{subfigure}{0.8\textwidth}
		\centering
		\includegraphics[width=\textwidth]{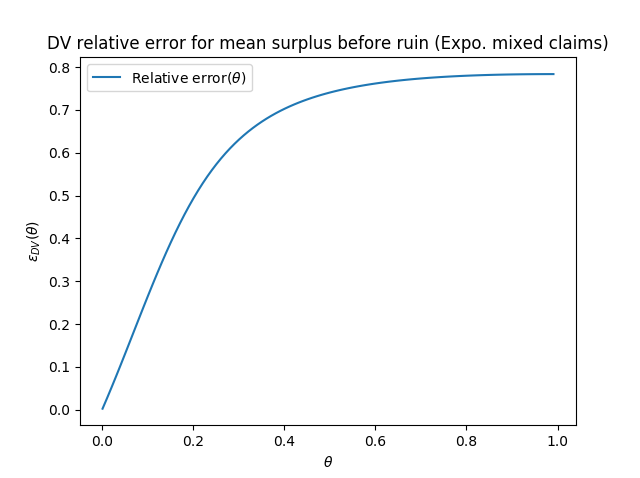}
		\caption{Relative error, as a function of $\t$ (with $u=100$)}
	\end{subfigure}
	\caption{De Vylder's approximation relative error for the average surplus before ruin. The claims density is the same as in Figure \ref{fig:DV:approx:surplus}.}
	\label{fig:DV:err:surplus}
\end{figure}

\subsection{Proof of Theorem \ref{thm:DV:err:surplus}}

\begin{proof}[Proof of Theorem \ref{thm:DV:err:surplus}]
	From \cite[identities (5.3) and (5.5)]{lin:will:00}, we know that
	\begin{align*}
		s_j(u)&=\frac{1}{\t\psi(u)}\left(\int_0^u\psi(u-x)x^jf_I(x)\dx+\int_u^\infty x^jf_I(x)\dx\right)-\frac{m_{j+1}}{(j+1)m_1\t}\\
		&=\frac{1}{\t\psi(u)}\left(\int_0^u(\psi(u-x)-\psi(u))x^jf_I(x)\dx+\int_u^\infty (1-\psi(u))x^jf_I(x)\dx\right).
	\end{align*}
	By Lemma \ref{lem:psi:exp:err}, we get
	\begin{align}
	s_j(u)&=\frac{1+\O(\t)}{\t}\left(\int_0^u(\exp(Rx)-1)x^jf_I(x)\dx+\int_u^\infty (\exp(Ru)-1-\O(\t))x^jf_I(x)\dx\right)\label{eq:sj}\\
	&\leq \frac{1+\O(\t)}{\t}\left(\int_0^\infty Rx\exp(Rx)x^jf_I(x)\dx+\int_0^\infty Rx\exp(Rx) x^jf_I(x)\dx+\O(\t)\right)\nonumber\\
	&=\frac{1+\O(\t)}{\t}\O(R+R+\O(\t))\nonumber\\
	&=\O(1+\t).\nonumber
	\end{align}
	Similarly, we have
	\begin{equation*}
	\sDV_j(u)=\O(1+\tDV)=\O(1+\t).
	\end{equation*}
Thus,
\begin{equation*}
	s_j(u)-\sDV_j(u)=\O(1+\t)=\O(1).
\end{equation*}
For a lower bound on $s_j(u)$, it is clear (by applying the inequality $\exp(Rx)-1\geq Rx$ to \eqref{eq:sj}) that
$$s_j(u)\geq\frac{1+\O(\t)}{\t}C(R\indic_{u\geq 1}+Ru\indic_{u\leq 1}+\O(\t))=C'(1+\O(\t))$$
(with positive constants $C$ and $C'$). Therefore,
\begin{equation*}
\frac{s_j(u)-\sDV_j(u)}{s_j(u)}=\O(\frac{1}{C'+\O(\t)})=\O(1).
\end{equation*}
\end{proof}

\section{Conclusion}\label{sec:concl}
One has to be careful when using De Vylder type approximations, even in a practical context. In the presence of a sufficiently small safety loading coefficient $\t$, our estimates show that the accuracy is good when approximating the ruin probability if the initial reserve $u$ is not too large, and when approximating the moments of the time of ruin if $u$ is not too small (otherwise, the relative errors blow up). However, the accuracy is generally poor when approximating the moments of either the deficit at ruin or the surplus before ruin, which is paradoxical (since De Vylder's approximation fits moments of the surplus process).\\
To summarize and illustrate once more our conclusions, Table \ref{tab:all} compares the numerical values of the relative errors of all considered De Vylder approximations, carried out on one common example of exponentially mixed claims.\\
We have not managed to derive general lower error bounds for De Vylder type approximations, which may constitute a subject for future research.
\begin{table}[!htbp]
\centering
\begin{tabular}{|c|c|c|c|c|}
	\hline
	Relative errors & for $\psi(u)$ & for $t_1(u)$  & for $d_1(u)$ & for $s_1(u)$\\
	\hline
	with $u=0$ & 7\% & 81\% & 81\% & 81\%\\
	\hline
	with $u=100$ & 4\% & 3\% & 16\% & 26\%\\
	\hline
 with $u=200$ & 14\% & 5\% & 16\% & 26\% \\
 \hline
\end{tabular}
\caption{An example of De Vylder approximation relative errors. The claims are exponentially mixed with density $a\beta_1\exp(-\beta_1x)+(1-a)\beta_2\exp(-\beta_2x)$, with $a=0.01$, $\beta_1=0.1$, $\beta_2=0.6$, and $\theta=10\%$.}
\label{tab:all}
\end{table}

\medskip

\printbibliography


\end{document}